\documentclass[11pt]{amsart}
\usepackage{amssymb, amsmath}
\usepackage{amsthm}
\usepackage{amscd}
\usepackage{graphicx}
\newtheorem{theorem}{Theorem}
\newtheorem{lemma}[theorem]{Lemma}

\newtheorem{question}[theorem]{Question}
\newtheorem{corollary}[theorem]{Corollary}
\newtheorem{proposition}[theorem]{Proposition}
\newtheorem{conjecture}[theorem]{Conjecture}
\theoremstyle{definition}

\newtheorem*{defn}{Definition}

\newtheorem*{example}{Example}
\newtheorem*{remark}{Remark}
\newtheorem*{acknowledgement}{Acknowledgement}

\title[$\pi_1$ of finite volume negatively curved $4$-manifold is not a $3$-manifold group]{Fundamental groups of finite volume,  bounded negatively curved $4$-manifolds are not $3$-manifold groups}
\author{Grigori Avramidi, T. T$\hat{\mathrm{a}}$m Nguy$\tilde{\hat{\mathrm{e}}}$n-Phan, Yunhui Wu}

\address{Grigori Avramidi\\
1027 Emerson Ave\\
Salt Lake City, UT 84105}
\email{gavramid@math.utah.edu}

\address{T. T$\hat{\mathrm{a}}$m Nguy$\tilde{\hat{\mathrm{e}}}$n-Phan\\
Department of Mathematics\\
Binghamton University\\
State University of New York\\
4400 Vestal Parkway East
Binghamton, NY 13902}
\email{tam@math.binghamton.edu}

\address{Yunhui Wu\\
Department of Mathematics\\
 Rice University, 6100 Main St\\
 Houston, TX 77005}
\email{yw22@rice.edu}

\def\ra{\rightarrow}

\def\beqa{\begin{eqnarray}}
\def\eeqa{\end{eqnarray}}
\def\beqa{\begin{eqnarray}}
\def\eeqa{\end{eqnarray}}
\def\to{\mapsto}

\DeclareMathOperator{\vol}{Vol}

\DeclareMathOperator{\actdim}{ad}

\def\Z{\mathbb{Z}}

\def\M{\overline{M}}

\input xy
\xyoption{all}
\begin{document}
\begin{abstract}
We study noncompact, complete, finite volume, Riemannian $4$-manifolds $M$ with sectional curvature $-1<K<0$. We prove that $\pi_1 M$ cannot be a $3$-manifold group. A classical theorem of Gromov says that $M$ is homeomorphic to the interior of a compact manifold $\M$ with boundary $\partial\overline M$. We show that for each $\pi_1$-injective boundary component $C$ of $\M$, the map $i_*$ induced by inclusion $i\colon C\rightarrow \M$ has infinite index image $i_*(\pi_1 C)$ in $\pi_1 \M$. We also prove that $M$ cannot be homotoped to be contained in $\partial\M$. 
\end{abstract}
\maketitle

\section{Introduction}
Let $M$ be a connected, complete, Riemannian manifold with finite volume and bounded sectional curvature $-1<K<0$. We study such manifolds $M$ and their fundamental groups in dimension $4$. The main result of this paper is the following.
\begin{theorem}\label{pi_1(M)}
Let $M$ be a complete, finite volume, Riemannian $4$-manifold of bounded negative curvature $-1<K<0$. Then $\pi_1(M)$ is not the fundamental group of a $3$-manifold.
\end{theorem}
The fundamental group $\pi_1(M)$ is torsion free because $M$ is aspherical, so if $\pi_1(M)$ acts properly discontinuously on a $3$-manifold $N^3$, then the quotient $N/\pi_1(M)$ is a manifold. Theorem \ref{pi_1(M)} implies that $\pi_1(M)$ does not act properly discontinuously on any manifold of dimension less than $4$. Therefore an immediate corollary of the above theorem is that the action dimension of $M$ is equal to $4$. (Recall that the \emph{action dimension} of a group $\Gamma$ is the smallest dimension $d$ for which there is a contractible manifold $X^d$ on which $\Gamma$ acts properly discontinuously.) 

The proof of Theorem \ref{pi_1(M)} is particular for dimension four and does not use any technique that was developed for computing the action dimension. 
We combine known results in pre-Geometrization $3$-manifold theory, the generalized Gauss-Bonnet theorem, and the $L^2$-cohomology theory of Cheeger and Gromov. 

Any finitely generated group is the fundamental group of a $4$-manifold, so the above phenomenon (the fundamental group $\pi_1(M)$ is not the fundamental group of a lower dimensional manifold) does not occur in dimension greater than four. However, it is reasonable to ask whether $\pi_1(M)$ can be the fundamental group of an  aspherical manifold of dimension less than $\dim M$.
\begin{question}\label{ad}
Let $M^n$ be a complete, finite volume, Riemannian $n$-manifold of bounded nonpositive curvature $-1<K\leq 0$. Assume that the fundamental group $\pi_1(M^n)$ is finitely generated. Is the action dimension of $\pi_1(M^n)$ equal to $n$?
\end{question}
The answer to the above question is clearly yes for $n=2$. It is also yes for $n=3$. We will give a short proof of this in the appendix (Section \ref{dim 3}). For general $n$, the answer is yes for the locally symmetric manifolds by a result of Bestvina and Feighn (\cite{bestvinafeighn}), but not so much is known for nonpositively curved manifolds whose universal covers have less symmetry. We observe that if the curvature is pinched between two negative constants or, more generally, if the universal covering space $\widetilde{M}$ of $M$ is a \emph{visibility} manifold (as defined by Eberlein and O'Neill \cite{EO}), then the action dimension of $\pi_1(M^n)$ is equal to $n$. 
Recall that $\widetilde{M}$ is a \emph{visibility} manifold or is said to satisfy the visibility axiom if for any two distinct points $x,y$ in the visual boundary of $\widetilde{M}(\infty)$ there exists a geodesic  $c\colon\mathbb{R}\rightarrow \widetilde{M}$ such that $c(-\infty)=x$ and $c(+\infty)=y$. If $M$ has pinched negative curvature $-1<K<-\varepsilon<0$, then $\widetilde{M}$ is a visibility manifold. 
\begin{proposition}\label{adovm}
Let $M$ be a complete, finite volume, Riemannian $n$-manifold of bounded nonpositive curvature $-1<K\leq 0$. Suppose that the universal cover $\widetilde{M}$ is a visibility manifold. Then the action dimension of $\pi_1(M)$ is $\actdim(\pi_{1}(M))=n$. 
\end{proposition}


\bigskip
\noindent
\textbf{Peripheral topology of $M$ when $n=4$.}
By a classical theorem of Gromov, $M$ is a the interior of a compact manifold $\overline{M}$ with boundary $\partial\overline{M}$. A general question is which subsets of $M$ can(not) be ``pushed to infinity", i.e. homotoped to be in a collar neighborhood of $\partial\overline{M}$. Another corollary of Theorem \ref{pi_1(M)} is the following. 
\begin{corollary}\label{4d core}
Let $M^4$ be a noncompact, complete, finite volume, Riemannian manifold with sectional curvature $-1 < K <0$. Then $M$ can not be homotoped to be contained in a collar neighborhood of $\partial\M$.
\end{corollary}
We have the following conjecture for general dimension $n$.
\begin{conjecture}\label{nonperipheral core}
Let $M$ be the interior of a compact manifold $\M$ with boundary $\partial\M$. Suppose that $M$ has a complete, finite volume Riemannian metric $g$ with sectional curvature $-1 \leq K(g) \leq 0$. Then $M$ can not be homotoped to be contained in a collar neighborhood of $\partial\M$.
\end{conjecture}

\begin{remark}
Conjecture \ref{nonperipheral core} obviously holds when $M$ has dimension two. It also holds when $M$ has dimension three. We will prove this in the appendix. If we do not require the curvature to be bounded, then both Question \ref{ad} and Conjecture \ref{nonperipheral core} are false in dimension three (\cite{Phan}).
\end{remark}

Let $i \colon \partial\overline{M} \longrightarrow \overline{M}$ be the inclusion map. Other natural questions to ask are whether the induced map  $i_*$ on the fundamental group is injective, or if $i_*$ can be surjective. In \cite{buyalo} Buyalo proved that for one of the examples constructed by Abresch and Schr\"oder (\cite{abreschschroeder}) the map $i$ is not $\pi_1$-injective. It is not known if there is an example where $i_*\colon\pi_1(\partial \M) \rightarrow\pi_1(\M)$ is surjective. 

\begin{question}\label{wfc}
Let $M$ be a noncompact, complete, finite volume, Riemannian manifold with sectional curvature $-1 < K <0$. And let $C$ be a component of $\partial \M$. Let $i\colon C \hookrightarrow M$ be the inclusion. Is $[\pi_1(M)\colon i_*(\pi_1(C))]>1$?
\end{question}

The answer to this question is yes if $M$ has pinched negative sectional curvature $-1< K<-a^2 <0$, or more generally, when $M$ has bounded nonpositive sectional curvature $-1 <K\leq 0$ and the universal cover $\widetilde{M}$ is a visibility manifold. The following corollary addresses Question \ref{wfc} in dimension four for a special case where we put a restriction on the ends of $M$.
\begin{corollary}\label{infind}
Let $M^4$ be a noncompact, complete, finite volume, Riemannian manifold with sectional curvature $-1 < K <0$. Then for each $\pi_1$-injective end $C^{3}\times \mathbb{R}^{\geq 0}$, we have
\[[\pi_1(M)\colon\pi_1(C)]=\infty.\]
\end{corollary}

If we ignore base points, the question on surjectivity of $i_*$ can be replaced by asking whether any loop in $M$ can be homotoped to be in $\partial \M$. It is a conjecture of Farb (\cite{Phan})  that this cannot always be done even when we replace the curvature condition $-1<K<0$ of $M$ by the conditions $-1<K\leq 0$, and $M$ has \emph{geometric rank} equal to one if we assume $M$ is the interior of a compact manifold with boundary. Recall that a nonpositively curved manifold has \emph{geometric rank one} if there is a geodesic in its universal cover that does not bound an infinitesimal half flat.

This paper is organized as follows. In Section \ref{Euler} we recall some facts from pre-Geometrization $3$-manifold theory and prove that the Euler characteristic of an infinite, torsion free, finitely generated $3$-manifold group is nonpositive. In Section \ref{proof of main theorem} we prove Theorem \ref{pi_1(M)}.  In Section \ref{action dimension of visible manifolds} we prove Proposition \ref{adovm}. In Section \ref{peripheral topology}, we prove Corollary \ref{4d core} and Corollary \ref{infind}. The appendix discusses the above questions for dimension three. 

\begin{acknowledgement}
We would like to thank the Cornell 50th Topology Festival for hosting us in 2012 when this work started. The first two authors would like to thank the Math Department of Rice University for sponsoring a visit during which this work was finished. The first two authors would also like to thank Southwest airlines for keeping them bored and happy during an extended flight delay. This led to useful ideas that later made their way into the paper.
\end{acknowledgement}

\section{Euler characteristics of finitely generated $3$-manifold groups}\label{Euler}
In this section we recall some pre-Geometrization facts about $3$-manifolds and put them together to conclude that finitely generated $3$-manifold groups have non-positive Euler characteristics. 

A connected $3$-manifold $W$ is {\it prime} if it cannot be expressed as a nontrivial connect sum.
It is called {\it irreducible} if every embedded $2$-sphere bounds a $3$-ball. The only prime $3$-manifold that is not irreducible is $\mathbb{S}^1\times \mathbb{S}^2$ \cite[Proposition 1.4]{hatcher3manifolds}. Manifolds can be decomposed into prime components.

\bigskip
\noindent
\textbf{Prime Decomposition for $3$-manifolds}  \cite[1.5]{hatcher3manifolds}\label{pdt}
Let $W$ be a closed, orientable 3-manifold. Then there is a decomposition $W=N_1\#\dots\#N_i$ with the $N_i$ prime.

\bigskip
\noindent
\textbf{Sphere theorem}  (\cite[3.8]{hatcher3manifolds})
Let $W$ be a connected, orientable $3$-manifold. If $\pi_2W\not=0$ then there is an embedded $2$-sphere in $W$ representing a nontrivial element of $\pi_2W$.
A consequence is that irreducible manifolds with infinite fundamental groups are aspherical. 

If $\pi_1W$ is torsion free, we get the following amplification of the prime decomposition theorem. 
\begin{theorem}\label{mpdt}
Let $W$ be an orientable, closed, 3-manifold with torsion free fundamental group $\pi_1W$. Then, there is a closed, simply connected manifold $\Sigma$, a nonnegative integer $m$ and closed, oriented aspherical manifolds $\{N_{i}\}_{i=1}^n$ such that $W$ is homeomorphic to the connected sum 
\[\Sigma\#\underbrace{(\mathbb{S}^1\times \mathbb{S}^2)\#\cdots \#(\mathbb{S}^1\times \mathbb{S}^2)}_{m}\#N_{1}\#\cdots \#N_{n}.\] 
Consequently the fundamental group of $W$ decomposes as
\[\pi_{1}(W)\cong \underbrace{\mathbb{Z}*\cdots *\mathbb{Z}}_{m}*\pi_{1}(N_{1})*\cdots *\pi_{1}(N_{n}).\]  
\end{theorem}
\begin{remark}
The Geometrization theorem implies that $\Sigma$ is the $3$-sphere, but we do not use this. 
\end{remark}
To deal with nonclosed $3$-manifolds, we will also need Peter Scott's tameness theorem.   
\begin{theorem}[Scott core theorem, \cite{kap-hdg} 4.124]
If $W$ is a connected $3$-manifold with finitely generated fundamental group $\pi_1W$, then there is a compact connected $3$-dimensional submanifold-with-boundary $i:K\hookrightarrow W$ for which the inclusion $i$ defines an isomorphism of fundamental groups.  
\end{theorem}
The submanifold $K$ is called the {\it Scott core} of $W$.

\begin{theorem}\label{3vs4}
Suppose that $W$ is a $3$-manifold with infinite, torsion free, finitely generated fundamental group $\pi_1W$. Then $\chi(\pi_1W)\leq 0$. 
\end{theorem}
\begin{proof}
Passing to the orientation double cover if necessary, we may assume that $W$ is orientable (doing this replaces the fundamental group by an index two subgroup, and thus scales the Euler characteristic by $2$). We begin by proving the theorem in the case when $W$ is closed. 
In this case we can decompose $\pi_1W$ as
\[\pi_{1}(W)\cong \underbrace{\mathbb{Z}*\cdots *\mathbb{Z}}_{m}*\pi_{1}(N_{1})*\cdots *\pi_{1}(N_{n}),\]  
where the $N_i$ are closed aspherical $3$-manifolds. 
The right hand side is the fundamental group of the aspherical complex $K(\pi_1W,1)=(\vee^m\mathbb{S}^1)\vee N_1\vee ...\vee N_n$ obtained
by taking the wedge of $m$ circles and the $n$ aspherical manifolds $N_i$. The Euler characteristic of this wedge is 

%
\[\chi (\pi_1(W))=\chi(K(\pi_1W,1)) = m \chi (\mathbb{S}^1) + \chi  (N_1) + ... + \chi (N_n) - (m+n-1) = 1-m-n.\]
Here, the last equality follows because the Euler characteristic of $S^1$ and of any closed $3$-manifold vanishes. Finally, $\pi_1W$ is infinite, so $m+n\geq 1$ which implies that $\chi(\pi_1W)\leq 0$.  

Now, we deal with the general case when $W$ may not be closed but $\pi_1W$ is finitely generated. Its Scott core is a compact $3$-manifold-with-boundary $(K,\partial K)$ with fundamental group $\pi_1W$. 
The boundary $\partial K$ is a disjoint union of closed orientable surfaces. Label the spheres in the boundary by $S^2_1,\dots, S^2_n$. Fill in these spheres by gluing in $3$-balls to get a new $3$-manifold
\begin{equation}
K':=K\cup D_1^3\cup\cdots\cup D_n^3
\end{equation}
with the same fundamental group $\pi_1K'=\pi_1W$. If $K'$ is closed then the previous paragraph shows $\chi(\pi_1W)\leq 0$ and we are done.  
So, suppose that $K'$ is not closed. Then $b_3(K')=0$. Moreover $b_0(K')=b_0(\pi_1W)=1$ and $b_1(K')=b_1(\pi_1W)$ since $K'$ and $W$ have the same fundamental group. Further, the Hopf exact sequence \cite{Hopf42}
$$
\pi_2(K')\ra H_2(K')\ra H_2(\pi_1K')\ra 0
$$
implies $H_2(K')\ra H_2(\pi_1K')=H_2(\pi_1W)$ is onto so that $b_2(K')\geq b_2(\pi_1W)$.
Putting these results together, we get
\begin{eqnarray}
\chi(K')&=&-0+b_2(K')-b_1(K')+1,\\
&\geq &-b_3(\pi_1W)+b_2(\pi_1W)-b_1(\pi_1W)+1,\\
&=&\chi(\pi_1W).
\end{eqnarray} 
Let $D:=K'\cup_{\partial K'}K'$ be the double of $K'$ along its boundary. The double is a closed $3$-manifold, so its Euler characteristic vanishes and we get
\begin{equation}
0=\chi(D)=2\chi(K')-\chi(\partial K'),
\end{equation}
which implies
\begin{equation}
\chi(\pi_1W)\leq\chi(K')={1\over 2}\chi(\partial K')\leq 0,
\end{equation}
where the last inequality follows because the boundary $\partial K'$ has no $2$-spheres. 
\end{proof}


\section{Negatively curved $4$-manifold groups are not $3$-manifold groups}\label{proof of main theorem}

The proof of Theorem \ref{pi_1(M)} is an amalgamation of several well known results together with the theorem from the previous section. Let $M$ be a $4$-manifold with a complete, finite volume, Riemannian metric of bounded negative curvature
($-1<K<0$). The Euler class of the
tangent bundle of $M$ has an explicit description as a differential form $e(TM)$ built out of the Riemann curvature tensor. That description leads to the following theorem.
\begin{theorem}[\cite{bishopgoldberg} 4.1]
Let $M$ be a complete Riemannian $4$-manifold of negative curvature. Then the Euler form is a positive multiple of the volume form, i.e. $e(TM)=\phi\cdot d\vol_M$ for some positive function $\phi:M\ra\mathbb R^+$. Moreover, if the sectional curvatures of $M$ are uniformly bounded ($-1<K<0$) then the function $\phi$ is also bounded.   
\end{theorem}
Integrating, one gets the following consequence. 
\begin{corollary}
If $M$ is a complete, finite volume, Riemannian $4$-manifold of bounded negative curvature $(-1<K<0)$ then 
\begin{equation}
\infty>\int_Me(TM)>0.
\end{equation}
\end{corollary}
For closed manifolds, the Chern-Gauss-Bonnet theorem implies that the integral of the Euler class is equal to the Euler characteristic $\chi(M)$.
The $L^2$-cohomology theory developed by Cheeger and Gromov extends this result to a class of noncompact manifolds satisfying the following
\emph{bounded geometry} condition.
\begin{defn}
A Riemannian manifold $N$ has \emph{bounded geometry} if its volume is finite, its sectional curvatures are uniformly bounded ($|K|\leq 1$) and the injectivity radius of the universal cover $\widetilde N$ is bounded below by a positive constant ($\mathrm{inj Rad}_{\widetilde N}>\varepsilon>0$).  
\end{defn}
Cheeger and Gromov proved that for manifolds satisfying their bounded geometry condition, the $L^2$-Euler characteristic $\chi^{(2)}(M)$ is finite
and can be computed as the integral of the Euler form. 
\begin{theorem}[Cheeger-Gromov, \cite{cheegergromov} 0.14]
If the Riemannian manifold $N$ has bounded geometry then 
\begin{equation}
\int_Ne(TN)=\chi^{(2)}(N).
\end{equation}
\end{theorem}
Gromov's finiteness theorem (\cite{gromovnegativelycurved} 1.8) implies that the manifold $M$ is the interior of a compact manifold-with-boundary $(\overline M,\partial\overline M).$
Moreover, it has bounded geometry (since the injectivity radius of the universal cover $\widetilde M$ is infinite). In this case, the Euler characteristic $\chi(M)$ and $L^2$-Euler characteristic $\chi^{(2)}(M)$ are well-defined and equal.
\begin{proposition}[Euler Poincare formula, \cite{luck} 1.36(2)]
If $X$ is a finite CW-complex then 
\begin{equation}
\chi(X)=\chi^{(2)}(X).
\end{equation}
\end{proposition}

\subsection*{Proof of Theorem \ref{pi_1(M)}}
By Gromov's finiteness theorem $\pi_1M$ is the fundamental group of a finite complex, so 
$$
\chi(\pi_1M)=\chi(M)=\chi^{(2)}(M)=\int_Me(TM)>0.
$$
Since $\pi_1M$ is finitely generated, Theorem \ref{3vs4} implies $\pi_1M$ is not the fundamental group of a $3$-manifold. 
\begin{remark}
In \cite{McMullen00} C. McMullen showed that for a surface $S_{g,n}$ of genus $g$ and $n$ punctures, the Euler characteristic of the mapping class group $\mathbb{M}_{g,n}$ of $S_{g,n}$ is positive if $6g+2n-6=4k$ for some positive integer $k$, by  constructing a K\"ahler-hyperbolic Riemannian metric,  which has bounded geometry, on the moduli space $\mathbb{M}_{g,n}$ of surfaces $S_{g,n}$. By Theorem \ref{3vs4}, it follows that any finite index torsion-free subgroup of the mapping class group $Mod_{g,n}$ cannot be isomorphic to the fundamental group of a 3-manifold if $6g+2n-6=4k$ for some positive integer $k$.
\end{remark}

\section{Action dimension of manifolds with visibility universal cover}\label{action dimension of visible manifolds}
Before we prove Proposition \ref{adovm}, let us recall two facts on the structure of $M$. Firstly Theorem 3.1 in \cite{Eberlein80} tells us that $M$ has only finitely many ends and each end is homeomorphic to a codimension-one closed submanifold, called a cross-section, cross the real line. In particular the fundamental group of $M$ is finitely generated. Secondly Lemma 3.1g in \cite{Eberlein80} tells us that each cross section is a quotient space of a horosphere, which is aspherical. Moreover, the fundamental group of each cross section is a subgroup of the fundamental group of $M$ which consists of parabolic isometries and fixes a common point in the visual boundary of the universal covering space of $M$. One can refer to \cite{Eberlein80} for the notations and concepts. The following lemma will be applied.

\begin{lemma}\label{cii}
Let $C$ be any cross section of $M$. Then $[\pi_{1}(M):\pi_{1}(C)]=\infty$.
\end{lemma}
\begin{proof}
Since $M$ is visible and has finite volume, $M$ must has a nontrivial closed geodesic (see Theorem 2.13 in \cite{Ballmann82}) which represents a hyperbolic isometry of the universal covering space of $M^n$. We denote this hyperbolic isometry by $\phi$. Firstly Lemma 3.1g in \cite{Eberlein80} tells us that $\pi_{1}(C)$ consists of parabolic elements except the unit. Thus, $\phi \notin \pi_{1}(C)$.

Assume that the conclusion is wrong. Then there exists a non-zero integer $m \in \mathbb{Z}$ such that $\phi^{m} \in \pi_{1}(C)$. In particular $\phi^m$ is parabolic, which is impossible since a well-known fact in CAT(0) geometry says that $\phi^m$ is parabolic if and only if $\phi$ is parabolic (one can see Theorem 6.8 in \cite{BH}).  
\end{proof}
 
Now we are ready to prove Proposition \ref{adovm}.

\begin{proof}[Proof of Proposition \ref{adovm}]
If $M$ is closed, the conclusion is obvious. So we assume that $M$ is an open manifold. Firstly since the fundamental group $\pi_{1}(M)$ acts freely on the universal covering space of $M$, $\actdim(\pi_{1}(M))\leq n=dim(M)$. 

Assume that conclusion is not true. Then we have $\actdim(\pi_{1}(M)) \leq n-1$, which means that there exists a contractible manifold $N^{n-1}$ of dimension less than or equal to $n-1$ on which $\pi_{1}(M)$ acts properly discontinuously. Since $\pi_{1}(M)$ is torsion-free (because $M$ has nonpositive sectional curvature), the group $\pi_{1}(M)$ acts freely on $N^{n-1}$. In particular, the quotient space $N^{n-1}/\pi_{1}(M)$ is a manifold whose fundamental group is isomorphic to $\pi_{1}(M)$.

Let $C$ be a cross section in Lemma \ref{cii}. Since $\pi_{1}(C)$ is a subgroup of $\pi_{1}(M)$, we can find a covering space $\overline{N}$ of $N^{n-1}/\pi_{1}(M)$ such that $\pi_{1}(\overline{N})=\pi_{1}(C)$. Since $C$ is a closed aspherical manifold of dimension $n-1$, the top dimensional homology group $H_{n-1}(C)=H_{n-1}(\pi_{1}(C))$ does not vanish (we use $\mathbb{Z}_{2}$ as coefficients if $C$ is unoriented, otherwise we use $\mathbb{Z}$ as coefficients). Thus $\overline{N}$ is compact because $H_{n-1}(\overline{N})=H_{n-1}(\pi_{1}(\overline{N}))=H_{n-1}(\pi_{1}(C))$ which does not vanish. Since $\overline{N}$ is a covering space of $N^{n-1}/\pi_{1}(M)$ and compact, the covering must be a finite covering, which means that $[\pi_{1}(M):\pi_{1}(\overline{N})]<\infty$. That is
\[[\pi_{1}(M):\pi_{1}(C)]<\infty\]
which contradicts Lemma \ref{cii}.  
\end{proof}

\begin{remark}
Generally the action dimension of a group may be not preserved by subgroups of finite index (see \cite{bestvinafeighn}). However, the action dimension of $\pi_{1}(M)$ in Proposition \ref{adovm} is preserved by any subgroup $\Gamma$ of finite index because we can do the same argument for $\tilde{M}/ \Gamma$ where $\widetilde{M}$ is the universal covering space of $M$.
\end{remark}  

\begin{remark}
Even for complete finite volume manifolds with sectional curvatures $-1<K<0$ may not have visibility universal covers (one can see details in \cite{Phan, Wu12}).  
\end{remark}

\section{Peripheral topology}\label{peripheral topology}
Let $(M,\partial M)$ be a compact aspherical manifold-with-boundary. A subset $S\subset M$ is called {\it peripheral} if there is a homotopy $h_t:S\ra M$
with 
\begin{eqnarray*}
h_0&=&id_S,\\
h_1(S)&\subset&\partial M.
\end{eqnarray*} 
\begin{example}
The boundary $\partial M$ is obviously peripheral.   
If $M$ is a product $\partial M_0\times I$, or more generally if $M$ is the total space of a disk bundle with a non-vanishing section, then the entire $M$ is peripheral. On the other hand, we show below that a complete, finite volume Riemannian $4$-manifold of bounded negative curvature is not peripheral. 
\end{example}

\begin{proof}[Proof of Corollary \ref{4d core}]
Suppose that there is a homotopy $h_t\colon M \rightarrow M$ such that 
\begin{eqnarray*}
h_0&=&id_M,\\
h_1(M)&\subset&\partial \overline M.
\end{eqnarray*}
Let $C$ be the boundary component of $\partial\overline M$ that contains $h_1(M)$. Then the induced map ${h_1}_* \colon \pi_1(M) \rightarrow \pi_1(C)$ is injective. Therefore, $\pi_1(M)$ is isomorphic to the fundamental group of a $3$-manifold, namely the cover of $C$ that corresponds to ${h_1}_*(\pi_1(M))$. This contradicts Theorem \ref{pi_1(M)}. Hence such a homotopy $h_t$ does not exist, and $M$ is not peripheral.
\end{proof}


\begin{proof}[Proof of Corollary \ref{infind}]
Since $i_*$ is injective, $\pi_1(C)$ embeds as a subgroup of $\pi_1(M)$. Let $\widehat{M}:=\widetilde{M}/\pi_1(C)$ be the covering space of $M$  corresponding to $\pi_1(C)$.

Suppose that $[\pi_1(M)\colon\pi_1(C)]<\infty$. Then $\widehat{M}$ is a finite cover of $M$. It follows that $\widehat{M}$ is a  noncompact, complete, finite volume, Riemannian manifold with sectional curvature $-1<K<0$. By Theorem \ref{pi_1(M)}, $\pi_1(\widehat{M})$ is not isomorphic to the fundamental group of $3$-manifold. But $\pi_1(\widehat{M}) \cong \pi_1(C)$. Since $C$ is a $3$-manifold, this is a contradiction. So    $[\pi_1(M)\colon\pi_1(C)]=\infty$.
\end{proof}

\section{Appendix: some remarks on  dimension $3$}\label{dim 3}
In this section we discuss Question \ref{ad} and Question \ref{wfc} for the case where $M$ is a complete, open, finite volume, Riemannian $3$-manifold with sectional curvature $-1\leq K\leq 0$ and $\pi_1(M)$ is finitely generated. These include finite volume manifolds with $-1<K<0$.

\begin{remark} We have a few remarks.
\begin{itemize}
\item[a)] The manifold $M$ in Theorem \ref{adotm} is the interior of a compact manifold $\M$ with boundary $\partial\M$ by a theorem of Gromov and Schr\"oder (\cite[Corollary 2, p. 194]{ballmangromovschroeder}).
\item[b)] Let $C$ be a component of $\partial \M$. By \cite[Theorem 1.2]{ChGromov85}) the Euler characteristic of $C$ vanishes. Since $C$ is oriented and closed, it has to be the $2$-torus. 
\item[c)] Corollary 4.2 in \cite{BelePhan} implies $i_{*}:\pi_{1}(C)\to \pi_{1}(V)$ is injective. 
\end{itemize}
\end{remark}

\begin{lemma}\label{fgcnba}
Let $M$ be a $3$-dimensional orientable, complete, open, finite volume, Riemannian manifold with sectional curvature $-1\leq K \leq 0$. Suppose that $\pi_1(M)$ is finitely generated. Then $[\pi_{1}(M^3):\pi_{1}(C)]=\infty$. 
\end{lemma}

\begin{proof}
Since $i_*$ is injective, $\pi_1(C)$ embeds as a subgroup of $\pi_1(M)$. Let $\widehat{M}:=\widetilde{M}/\pi_1(C)$ be the covering space of $M$  corresponding to $\pi_1(C)$. Since $\pi_1(C) \cong \Z^2$, which is abelian, either there is a point $\sigma\in\partial_\infty\M$ such that $\pi_1(C)$ preserves the horospheres centered at $\sigma$, or $\pi_1(C)$ consists of commuting hyperbolic isometries that preserve a common flat (see Theorem 7.1 in \cite{BH}). In either case, $\widehat{M}$ does not have finite volume. Therefore $[\pi_1(M)\colon\pi_1(C)]=\infty$.
\end{proof}

We give an affirmative answer to Question \ref{ad} in dimension three. 

\begin{theorem}\label{adotm}
Let $M$ be a $3$-dimensional, complete, finite volume, Riemannian manifold with sectional curvature $-1\leq K \leq 0$. Suppose that $\pi_1(M)$ is finitely generated. Then the action dimension $\actdim(\pi_{1}(M))=3$.
\end{theorem}
 
\begin{proof}
Firstly we assume that $M$ is orientable. It is clear that $\actdim(\pi_{1}(M))\leq 3$ since $\pi_{1}(M^3)$ acts freely on the contractible manifold $\widetilde{M}$. If $M^3$ is closed, then $\actdim(\pi_{1}(M^3))=3$. Assume that $M^3$ is open. Suppose that $\pi_1(M)$ acts properly discontinuously on a contractible $2$-manifold, i.e. $\mathbb{R}^2$. Hence $\pi_1(M) \cong \pi_1(\Sigma)$ for some surface $\Sigma$. Since $\pi_1(M)$ contains a subgroup $\pi_1(C) \cong \mathbb{Z}^2$, it follows that $\Sigma$ must be either the $2$-torus or the Klein bottle.  In either case $\pi_1(C)$ has finite index in $\pi_1(M)$, which contradicts Lemma \ref{fgcnba}.

If $M$ is unorientable. Let $\overline{M}$ be the double covering space of $M$, which is orientable. From the definition it is clear that $\actdim(\pi_{1}(\overline{M}))\leq \actdim(\pi_{1}(M))\leq 3$. The argument in the paragraph above gives that $\actdim(\pi_{1}(\overline{M}))=3$. Therefore, $\actdim(\pi_{1}(M))=3$.   
\end{proof}

\bibliographystyle{amsplain}
\bibliography{locsym}

\end{document}